\documentclass[11pt,a4paper]{amsart}
\usepackage{amsfonts}
\usepackage{amsthm}
\usepackage{amsmath}
\usepackage{amscd}
\usepackage[foot]{amsaddr}
\usepackage{braket}
\usepackage[latin2]{inputenc}
\usepackage{t1enc}
\usepackage{tikz-cd}
\usepackage[mathscr]{eucal}
\usepackage{MnSymbol}
\usepackage{indentfirst}
\usepackage{graphicx}
\usepackage{graphics}
\usepackage{pict2e}
\usepackage{epic}
\numberwithin{equation}{section}
\usepackage[margin=2.9cm]{geometry}
\usepackage{epstopdf} 
\usepackage[breaklinks=true]{hyperref}
\usepackage[nameinlink]{cleveref}

\newtheorem{thm}{\textbf{Theorem}}[section]
\newtheorem*{thm*}{Theorem}
\newtheorem{lem}[thm]{\textbf{Lemma}}
\newtheorem*{lem*}{Lemma}
\newtheorem{prop}[thm]{\textbf{Proposition}}

\theoremstyle{remark}
\newtheorem*{rem}{Remark}

\crefname{thm}{Theorem}{Theorems}
\Crefname{thm}{Theorem}{Theorems}
\crefname{lem}{Lemma}{Lemmas}
\Crefname{lem}{Lemma}{Lemmas}
\crefname{prop}{Proposition}{Propositions}
\Crefname{prop}{Proposition}{Propositions}
\crefname{cor}{Corollary}{Corollaries}
\Crefname{cor}{Corollary}{Corollaries}

\newcommand{\res}[1]{\textup{Res}_{#1}\hspace{3pt}}
\newcommand{\alg}{\text{alg}}

\begin{document}

\title{Characters in \textit{p}-adic Heisenberg and Lattice Vertex Operator Algebras}

\author[Daniel Barake]{Daniel Barake}
\author[Cameron Franc]{Cameron Franc}
\email{baraked@mcmaster.ca}
\email{franc@math.mcmaster.ca}

\begin{abstract}
We study characters of states in $p$-adic vertex operator algebras. In particular, we show that the image of the character map for both the $p$-adic Heisenberg and $p$-adic lattice vertex operator algebras contain infinitely-many non-classical $p$-adic modular forms which are not contained in the image of the algebraic character map.
\end{abstract}

\maketitle

\section{Introduction}

We study $p$-adic properties of certain vertex operator algebras. Motivated by both physical and number-theoretical methods, the authors of \cite{FM} introduce the study of $p$-adic VOAs which arise from a completion of the axioms for usual (algebraic) VOAs. Though the existence of $p$-adic variants of known VOAs such as the Virasoro, Monster and Heisenberg was central to this work, the authors demonstrate in Sections 9 and 10 that the character map (i.e. $1$-point correlation function or graded trace) on such VOAs has image in Serre's ring of $p$-adic modular forms $\mathfrak{M}_{p}$. This facet of the theory is at the heart of our current discussion. \par 
\vspace{10pt}
Following \cite{FM}, and denoting by $S_{\alg}$ the Heisenberg algebra, the character map is first extended to a surjective $\mathbb{Q}_{p}$-linear map
\begin{align}
\textbf{\textit{f}}: S_{\alg} &\to \mathbb{Q}_{p} [E_{2},E_{4},E_{6}] \label{padiccharactermap} \\
a &\mapsto \eta(q)Z(a,q) \nonumber
\end{align}
where $q = e^{2 \pi i z}$ and where $E_{2}$, $E_{4}$ and $E_{6}$ denote the Eisenstein series of weights $2$, $4$ and $6$ respectively, normalized such that they have constant term equal to $1$ (cf. \cite{MT,DMN}). Note above that $\textbf{\textit{f}}$ is a rescaling of the character map $Z(a,q)$ which ensures the $\eta (q)^{-1}$ factor is removed from the image. The $p$-adic completion of $\mathbb{Q}[E_{4},E_{6}]$ by the sup-norm yields $\mathfrak{M}_{p}$ (cf. \cite{pSerre}), and so motivated by this construction, states in the $p$-adic Heisenberg VOA $S_{1}$ are formed as $p$-adic limits of convergent sequences of states in $S_{\alg}$. Theorem 9.7 of \cite{FM} then establishes that $\textbf{\textit{f}}$ extends further to a map into $\mathfrak{M}_{p}$, thus obtaining $p$-adic modular forms as characters of states in $S_{1}$ amounts to making the following diagram commute:
\begin{center}
\begin{tikzcd}
S_{1} \arrow{r}{\textbf{\textit{f}}} & \mathfrak{M}_{p}  \\ S_{\alg} \arrow[r, tail, twoheadrightarrow, "\textbf{\textit{f}}"] \arrow[u, tail, hookrightarrow] & \mathbb{Q}_{p} [E_{2},E_{4},E_{6}] \arrow[u, tail, hookrightarrow]
\end{tikzcd}
\end{center}
Here, the upward arrows denote the respective completion functors. Unlike the algebraic case in which it is known that $\textbf{\textit{f}}$ is surjective on the ring of quasi-modular forms, it is yet undetermined whether every $p$-adic modular form arises as the character of some state from a $p$-adic VOA. It was established in Section 10 of \cite{FM} that the image of $\textbf{\textit{f}}$ on $S_{1}$ contains the $p$-adic Eisenstein series of weight $2$ whose $q$-expansion is given by
\begin{align}
G_{2}^{\star}(q) = \frac{p-1}{24} + \sum_{n \geq 1} \sigma^{\star}(n) q^{n} \label{p-adicmodformg2}
\end{align}
where $\sigma^{\star}(n)$ denotes the sum of all divisors of $n$ which are coprime to a prime $p$. It is then natural to speculate whether the entire $p$-adic Eisenstein family lies in the image of the character map on $S_{1}$. Indeed, by considering a broader set of states and extending the methods of \cite{FM}, in \Cref{Section: PadicCharacters} of this discussion we prove: \par 
\begin{thm} \label{Th: theorem1}
The image of the character map on the $p$-adic Heisenberg VOA $S_{1}$ contains infinitely-many non-classical $p$-adic modular forms. In particular, the image contains
\begin{align*}
2^{(t+1)/2}t(t-2)!!G_{2}^{(t-1)/2}(q)G_{l+1}^{\star}(q)
\end{align*}
which is not quasi-modular of level one, and where $t,l \geq 1$ are odd and $l \not\equiv -1 \mod p-1$. 
\end{thm}
The appearance of the factors $G_{2}^{(t-1)/2}$ above is also new and provides some clarity on the image of $\textbf{\textit{f}}$ on $S_{1}$. Note that setting $t=1$ above and normalizing returns the classical $p$-adic Eisenstein series $G_{k}^{\star}$ where $p-1 \nmid k$. \par 
\vspace{10pt}
In a similar direction, the character map on VOAs based on lattice theories $V_{\Lambda}$ (cf. \cite{FLM,LL}) is comparable to that of $S_{\alg}$ albeit with the emergence of theta series (cf. \cite{DMN}), and so it is feasible to suspect that results such as \Cref{Th: theorem1} are not unique to $S_{1}$. From \cite{FM}, one may deduce the existence of $p$-adic variants of $V_{\Lambda}$ formed from the $p$-adic completion with respect to the sup-norm. Such VOAs will be denoted by $\widehat{V}_{\Lambda}$, and we have an analogue of the character map seen in \cref{padiccharactermap}:
\begin{align}
    \textbf{\textit{f}}: V_{\Lambda} &\to \mathbb{Q}_{p} [E_{2},E_{4},E_{6}] \label{padiccharactermaplattice} \\
a &\mapsto \eta(q)^{d}Z(a,q) \nonumber
\end{align}
where $d$ is the rank of the lattice $\Lambda$. It will be clear in our discussion when working over $S_{\alg}$ or $V_{\Lambda}$, and so the use of $\textbf{\textit{f}}$ for both characters should cause no confusion. In this case, concrete examples of $p$-adic characters have yet to be constructed prior to this work. In \Cref{Section: padiccharactersinlatticevoas}, we build infinite families of states in $\widehat{V}_{\Lambda}$ which give rise to $p$-adic modular forms now in the sense of Katz \cite{Katz}. Our second theorem is:
\begin{thm} \label{Th: theorem2}
Let $\Lambda$ be an even unimodular lattice. The image of the character map on the $p$-adic lattice VOA $\widehat{V}_{\Lambda}$ contains infinitely-many non-classical Katz $p$-adic modular forms. In particular, the image contains the sum
\begin{align*}
2tG_{l+1}^{\star}(q)  \sum_{k \geq 0} \binom{t-1}{2k} \frac{(2k)!}{k!} G_{2}(q)^{k} \Theta_{\Lambda, t,k}(q)
\end{align*}
where $t,l \geq 1$ are odd with $l \not\equiv -1 \mod p-1$, and each $\Theta_{\Lambda, t,k}$ is quasi-modular. 
\end{thm}

The precise definition of $\Theta_{\Lambda, t,k}$ is given in the statement of \Cref{Prop: specificcharactermap2}. In this case, by setting $t=1$ and normalizing, we obtain the product $G_{k}^{\star}\Theta_{\Lambda}$ where $\Theta_{\Lambda}$ is the theta series for $\Lambda$ which is well-known to be modular since $\Lambda$ is even and unimodular. \par 
\vspace{10pt}
Both \Cref{Th: theorem1,Th: theorem2} demonstrate the need for an ameliorated theory better suited to the study of the $p$-adic character. The methods currently employed are analogous to those which were used by Serre in \cite{pSerre} and generalized in a geometric way by Katz in \cite{Katz}. Hence a geometric reformulation of the axioms for p-adic VOAs in view of Katz' theory of $p$-adic modular forms could serve as such a device.   \par 
\vspace{10pt}
Another possibility would be to develop a theory of Hecke operators arising from $S_{\alg}$ or $V_{\Lambda}$. For example, a realization of the operator $U_{p}$ as in \cite{pSerre} directly within a $p$-adic VOA could prove useful in assessing the surjectivity of the $p$-adic character on $\mathfrak{M}_{p}$. Though this is only hypothetical, the authors of \cite{Katsura}, \cite{Katsura2}, and \cite{Katsura3} have shown that such operators do arise algebraically in the context of conformal field theory, namely out of the Fock space of the Fermionic algebra. It is then an open question whether similar structures exist in $S_{\alg}$, $V_{\Lambda}$, or in any of their various incarnations. \par

\section{Background}

\subsection{Vertex Operator Algebras} \label{Subsection: VOAs}

For elementary notions in VOA theory, the reader is invited to consult \cite{LL}. A vertex operator algebra consists of a $\mathbb{Z}$-graded vector space, or Fock space,
\begin{align*}
V = \coprod_{n \in \mathbb{Z}} V_{(n)}
\end{align*}
where $\dim V_{(n)} < \infty$ for all $n$ and $V_{(n)} = 0$ for $n \ll 0$. We say the state $a \in V$ is homogeneous of weight $k$ if $a \in V_{(k)}$. The vertex operator linear map
\begin{align*}
Y: V &\to \text{End}(V)[[z,z^{-1}]] \\
a &\mapsto Y(a,z) = \sum_{n \in \mathbb{Z}} a(n)z^{-n-1}
\end{align*}
packages the modes $a(n) \in \text{End}(V)$ of the state $a$ into a formal Laurent series, and these satisfy the truncation condition $a(n)b = 0$ for $n \gg 0$ for $a,b \in V$. The vacuum state $\textbf{1} \in V_{(0)}$ satisfies $Y(\textbf{1},z) = 1$ and $\lim_{z \to 0} Y(a,z)\textbf{1} = a$. These are the vacuum and creation properties, respectively. The conformal state $\omega \in V_{(2)}$ has vertex operator given by
\begin{align*}
Y(\omega, z) = \sum_{n \in \mathbb{Z}} L(n) z^{-n-2} \hspace{5pt} \left( = \sum_{n \in \mathbb{Z}} \omega(n) z^{-n-1} \right)
\end{align*}
where the $L(i) \in \text{End}(V)$ generate a copy of the Virasoro algebra with central charge $c_{V} \in \mathbb{C}$. The $L(0)$-eigenspace decomposition of $V$ coincides with the grading of $V$, that is, $L(0)a=ka$ where $a \in V_{(k)}$, and the $L(-1)$-derivative property holds, i.e. $L(-1)\textbf{1}=0$. Finally, the Jacobi identity is satisfied for any $a,b \in V$ and triple of integers $(r,s,t)$:
\begin{align}
\sum_{i \geq 0} \binom{r}{i} (-1)^{i} &a ( s+r-i ) b (t+i) - (-1)^{r} \sum_{i \geq 0} \binom{r}{i} (-1)^{i} b (t+r-i) a (s+i) \label{jacobi} \\
&= \sum_{i \geq 0} \binom{s}{i} (a(r+i)b)(s+t-i). \nonumber
\end{align}
We shall also require the notion of a $V$-module $M$ in arguments involving lattice VOAs. Recall this is a $\mathbb{C}$-graded vector space
\begin{align*}
M = \coprod_{h \in \mathbb{C}} M_{(h)}
\end{align*}
such that $\dim M_{(h)} < \infty$ for all $h$ and $M_{(h)} = 0$ for $h \ll 0$. Here also, the $L(0)$-eigenspace decomposition coincides with the grading above. We call $a \in M$ homogeneous of weight $h$ if $a \in M_{(h)}$. We have an analogous vertex operator linear map
\begin{align*}
Y_{M}: V &\to \text{End}(M)[[z,z^{-1}]] \\
a &\mapsto Y_{M}(a,z) = \sum_{n \in \mathbb{Z}}a(n)z^{-n-1}
\end{align*}
where the truncation and vacuum conditions as well as the Jacobi identity (\cref{jacobi}) hold true for the above modes. \par 
\vspace{10pt}
Given $a \in V_{(k)}$ and $m \in \mathbb{Z}$, one shows that $a(n): V_{(m)} \to V_{(m+k-n-1)}$ (cf. \cite{MT}). Denote then the zero mode $o(a)$ of $a$ as $a(k-1)$ such that $o(a): V_{(m)} \to V_{(m)}$, and extend the definition to all of $V$ additively. By tracing the zero mode over each subspace $V_{(n)}$ (recall $\dim V_{(n)} < \infty$ for all $n \in \mathbb{Z}$) and packaging these into a generating function, we obtain the character
\begin{align*}
Z(a,q) = q^{-c_{V}/24} \sum_{n \in \mathbb{Z}} \text{Tr}_{V_{(n)}}o(a)q^{n}.
\end{align*}
In the case of $V$-modules, the character is defined in a natural way where the zero mode $o(a) = a(k-1) \in \text{End}(M)$ for $a \in V_{(k)}$ preserves each homogeneous space $M_{(h)}$. \par 

\subsection{\textit{p}-adic Vertex Operator Algebras}

The definition of a VOA given above can be easily generalized over any field $k$ of characteristic zero. In particular, it is perfectly reasonable to set $k = \mathbb{Q}_{p}$. All of the VOAs of interest in our discussion (Heisenberg modules and lattice VOAs) have bases defined over the $p$-adic integers $\mathbb{Z}_p$, such that the corresponding modes expressed in these bases are all $p$-adically integral, satisfying some mild hypotheses on boundedness of these operators. A $p$-adic VOA $\widehat{V}$ is then formed by completing and inverting $p$. One description is to first consider the $\mathbb{Z}_p$-module
\begin{align}
    \widehat{U} = \lim_{\substack{ \leftarrow \\ m }} U / p^{m} U
\end{align}
and then letting $\widehat{V} = \widehat{U}\otimes_{\mathbb{Z}_p}\mathbb{Q}_p$. Equivalently, $\widehat{V}$ can be described in terms of a completion with respect to a sup-norm arising from the integral basis in $U$. See Section 7 of \cite{FM} for details on these constructions. For our purposes, it will be enough to think of elements of $\widehat{V}$ as limits of convergent sequences of states in an underlying algebraic VOA, though the full axiomatic description of $\widehat{V}$ is given in \cite{FM}. \par 
\vspace{10pt}
For a concrete example, setting $V$ to be the Heisenberg algebra (which we will describe momentarily below as polynomial ring in countably many indeterminates $h(i)$) over $\mathbb{Q}_{p}$, let $S_{r}$ for $r \in \mathbb{R}_{> 0}$ denote the completion of $V$ with respect to the sup-norm
\begin{align}
    \left\lvert \sum_{I} a_{I}h^{I} \right\rvert_{r} = \sup_{I} \vert a_{I} \vert r^{I}
\end{align}
where $I$ is a finite multi-subset of $\mathbb{Z}_{< 0}$ with $\vert I \vert = -\sum_{i \in I} i$, and $h^{I} = \prod_{i \in I} h(i)$. Then Proposition 9.1 in \cite{FM} shows that elements of $S_{r}$ are of the form $\sum_{I} a_{I}h^{I}$ where
\begin{align}
    \lim_{\vert I \vert \to \infty} \vert a_{I} \vert r^{\vert I \vert} = 0.
\end{align}
In particular, with $r=1$, the ring $S_{1}$ has the structure of a $p$-adic VOA, and there is a nested chain of subpaces $S_{1} \supseteq S_{2} \supseteq S_{3} \supseteq \cdots $. \par

\subsection{\textit{p}-adic modular forms}

Following \cite{pSerre}, for a prime $p \geq 2$ we denote the $p$-adic weight space as
\begin{align*}
    X = \lim_{\substack{ \leftarrow \\ m }} \mathbb{Z} / p^{m} (p-1) \mathbb{Z} \cong \mathbb{Z}_{p} \times \mathbb{Z} / (p-1) \mathbb{Z}.
\end{align*}
Note that when $p = 2$, we have $X \cong \mathbb{Z}_{2}$. We say $k \in X$ is even if $k \in 2X$, and that $k$ is odd if it is not even. A $p$-adic modular form is then a formal series
\begin{align*}
    f = \sum_{n=0}^{\infty} a_{n} q^{n} \in \mathbb{Q}_{p}[[z]]
\end{align*}
such that $f$ is the $p$-adic limit of modular forms $f_{i}$ of weight $k_{i}$ having rational coefficients. The ring $\mathfrak{M}_{p}$ is then defined as the completion of $\mathbb{Q}_{p}[E_{4},E_{6}]$ by the sup-norm on the Fourier coefficients $a_{n}$, and $f \in \mathfrak{M}_{p}$ has weight equal to the $p$-adic limit of the $k_{i}$, which is an element of $X$. \par 
\vspace{10pt}
The ring $\mathfrak{M}_{p}$ contains $p$-adic analogues of classical Eisenstein series, and these take the form
\begin{align*}
    G_{k}^{\star}(q) = \frac{1}{2} \zeta_{p} (1-k) + \sum_{n=1}^{\infty} \sigma_{k-1}^{\star} (n) q^{n}
\end{align*}
where $\sigma_{k-1}^{\star}(n)$ is defined as in \cref{p-adicmodformg2} and $\zeta_{p}$ denotes the Kubota-Leopoldt $p$-adic zeta function (cf. \cite{KL}). For even integers $k \geq 2$, $G_{k}^{\star}$ is a modular form on the congruence subgroup $\Gamma_{0}(p)$. For example, taking the sequence $\lbrace G_{2+p^{m}(p-1)} \rbrace_{m \geq 0}$ of classical Eisenstein series as $m \to \infty$ for some $p \geq 2$ yields $G_{2}^{\star} \in \mathfrak{M}_{p}$ seen in \cref{p-adicmodformg2}. The images in \Cref{Th: theorem1,Th: theorem2} will be obtained in a similar manner.

\section{p-adic characters in the Heisenberg VOA} \label{Section: PadicCharacters}

\subsection{The Heisenberg VOA and the square-bracket formalism} \label{Subsection: HeisenbergandSquareBracket}

We briefly outline the construction of the Heisenberg VOA $S_{\alg}$. A rigorous treatment can be found in many texts such as \cite{FLM,LL,MT,FB}. For generalities on Lie algebras alongside their representations, we follow Chapter 9 of \cite{Carter}. Let $\mathfrak{h}$ be a finite-dimensional vector space viewed as an abelian Lie algebra. Then associated the (untwisted) affine Lie algebra
\begin{align*}
\widehat{\mathfrak{h}} = \left( \mathfrak{h} \otimes \mathbb{C}[t,t^{-1}] \right) \oplus \mathbb{C}\textbf{k}
\end{align*}
where $\textbf{k}$ is central. There are two graded subalgebras which we write with respect to the notation given in \cite{LL} (by weights) as:
\begin{align*}
&\widehat{\mathfrak{h}}_{+} = \mathfrak{h} \otimes t^{-1}\mathbb{C}[t^{-1}] \\
&\widehat{\mathfrak{h}}_{-} = \mathfrak{h} \otimes t \mathbb{C}[t].
\end{align*}
Define $\widehat{\mathfrak{h}}_{\leq 0} = \widehat{\mathfrak{h}}_{-} \oplus \left( \mathfrak{h} \oplus \textbf{k} \right)$, and denote by $\mathcal{U}( \cdot )$ the universal enveloping algebra. Let $\mathbb{C}\textbf{1}$ be the $1$-dimensional $\widehat{\mathfrak{h}}_{\leq 0}$-module where $\widehat{\mathfrak{h}}_{-}$ acts trivially on $\mathbb{C}$ and $\textbf{k}$ acts as the identity. Construct the induced module
\begin{align*}
S_{\alg} = \text{Ind}_{\mathcal{U}(\widehat{\mathfrak{h}}_{\leq 0})}^{\mathcal{U}(\widehat{\mathfrak{h}})} \mathbb{C}\textbf{1} = \mathcal{U}(\widehat{\mathfrak{h}}) \otimes_{\mathcal{U}(\widehat{\mathfrak{h}}_{\leq 0})} \mathbb{C}\textbf{1}
\end{align*}
where we write the action of $h \otimes t^{n}$ on $S_{\alg}$ as $h(n)$ for $n \in \mathbb{Z}$. Note that $S_{\alg} \cong S(\widehat{\mathfrak{h}}_{+})$, the symmetric algebra. That is, $\lbrace h(-1), h(-2), h(-3), \ldots \rbrace$ form a basis for $S_{\alg}$. Also note the $h(n)$ for $n \geq 1$ act on $S_{\alg}$ as $n \partial_{h(-n)}$, and $h(0)$ acts as zero. The Heisenberg vertex operator (or Heisenberg field) is defined as
\begin{align*}
h(z) = \sum_{n \in \mathbb{Z}} h(n)z^{-n-1} \in \text{End}(S_{\alg})[[z,z^{-1}]].
\end{align*}
This is no more than a generating function for the operators $h(n) \in \text{End}(S_{\alg})$. Hence we may view $S_{\alg}$ as being generated by the Fourier coefficients of the state $h(-1)\textbf{1}$ (see \cite{FB}). \par 
\vspace{10pt}
Instrumental to establishing the connection between VOAs and modular forms is the notion of transforming a VOA by way of a map $z \mapsto \phi (z)$, introduced first in \cite{Zhu}. An overview of this technique can be found in detail in \cite{MT} and in Chapter 5 of \cite{FB}. In particular, setting $\phi (z) = e^{z} - 1$ formulates the vertex operators as living on a torus rather than a sphere, and for $S_{\alg}$ these are given by
\begin{align}
Y(h(-1)\textbf{1}, e^{z}-1) e^{z} = \sum_{n \in \mathbb{Z}} h[n]z^{-n-1}. \label{transformedheisenberg}
\end{align}
This resulting ``square-bracket" VOA is equipped with a new grading which respects the grading on the space of quasi-modular forms. Multiplying \cref{transformedheisenberg} by $z^{n}$ we obtain
\begin{align}
h[n] &= \res{z}Y(h(-1)\textbf{1}, e^{z} -1) (e^{z})(z^{n}). \label{generalsquarebracketstate}
\end{align}
Note that $Y(h(-1)\textbf{1}, e^{z}-1)$ truncates when applied to any state in $S_{\alg}$, and that we can use a change of variables theorem (cf. eq. (1.1.3) of \cite{Zhu}) with $w = e^{z}-1$ to show
\begin{align}
h[n] = \res{w} h(w) \left( \log (w+1) \right)^{n} \label{squarebracketgeneral}
\end{align}
Observe that $h[0]=h(0)$. For our purposes, we will only require the explicit expression where $n=-1$ above. If we let $c_{0}=1$, $c_{i}=0$ for $i < 0$, and
\begin{align*}
c_{k} = \sum_{i \geq 1} \frac{(-1)^{i+1}}{i+1} c_{k-i}
\end{align*}
so that
\begin{align*}
c_{1} = \frac{1}{2}, \hspace{5pt} c_{2} = -\frac{1}{12}, \hspace{5pt} c_{3} = \frac{1}{24}, \hspace{5pt} c_{4} = -\frac{19}{720}, \hspace{5pt} \cdots
\end{align*}
then $c_{k} = \text{Coeff}_{w^{k-1}} ( \log (w+1) )^{-1}$ and we see that
\begin{align}
    h[-1] = h(-1) + \frac{1}{2} h(0) - \frac{1}{12} h(1) + \frac{1}{24} h(2) - \frac{19}{720} h(3) + \cdots  \label{squarebracketh-1}
\end{align}

\subsection{Some square-bracket states} \label{Subsection: somesquarebracketstates}

We construct as in \cite{FM} a family of square-bracket states in $S_{\alg}$ which map under $\textbf{\textit{f}}$ (recall \cref{padiccharactermap}) to an appropriate subspace of the ring of quasi-modular forms. In order to assess $p$-adic convergence in the subsequent \Cref{Lemma: stirlingconvergence,Lemma: bernoulliconvergence}, we must write these square-bracket states in the round-bracket formalism, that is, in terms of the basis $\lbrace h(-1),h(-2),h(-3) \ldots \rbrace$ of $S_{\alg}$. \par 
\begin{lem}\label{Lemma: correctionlemma1}
For an odd integer $t \geq 1$, we have
\begin{align}
h[-1]^{t}\textbf{1} = \sum_{k \geq 0} \binom{t}{2k} \frac{(2k)!}{k!(-24)^{k}}h(-1)^{t-2k}\textbf{1}. \label{correctionlemm1eq}
\end{align}
\end{lem}
\begin{proof}
We proceed by induction. Recall that in $S_{\alg}$, \cref{squarebracketh-1} shows that $h[-1]\textbf{1} = h(-1)\textbf{1}$ which agrees with the statement. Suppose the claim holds up to some $t>1$. Then applying $h[-1]$ to the right-hand side of \cref{correctionlemm1eq} gives
\begin{align*}
&\sum_{k \geq 0} \binom{t}{2k} \frac{(2k)!}{k!(-24)^{k}}h(-1)^{t-2k+1}\textbf{1} + \sum_{k \geq 0} \binom{t}{2k} \frac{2(2k)!(t-2k)}{k!(-24)^{k+1}} h(-1)^{t-2k-1}\textbf{1} \\
&= \sum_{k \geq 0} \binom{t}{2k} \frac{(2k)!}{k!(-24)^{k}}h(-1)^{t-2k+1}\textbf{1} + \sum_{k \geq 0} \binom{t}{2k+1} \frac{(2k+2)!}{(k+1)!(-24)^{k+1}} h(-1)^{t-2k-1}\textbf{1}. \\
\intertext{Re-indexing the second sum, we get}
&= h(-1)^{t+1} + \sum_{k \geq 1} \binom{t}{2k} \frac{(2k)!}{k!(-24)^{k}}h(-1)^{t-2k+1}\textbf{1} + \sum_{k \geq 1} \binom{t}{2k-1} \frac{(2k)!}{k!(-24)^{k}} h(-1)^{t-2k+1}\textbf{1},
\end{align*}
and by using Pascal's rule this equals $h[-1]^{t+1}\textbf{1}$ which completes the induction.
\end{proof}
\begin{rem}
It is interesting to note that the expression given in \Cref{Lemma: correctionlemma1} is related to generalized Hermite polynomials
\begin{align}
\text{He}_{n}^{\alpha}(x) = \left( x - \alpha \partial_{x} \right)^{n} \cdot 1 = e^{-\frac{\alpha (\partial_{x})^{2}}{2}}x^{n} \label{generalizedhermite}
\end{align}
for $n \geq 0$ and where $\alpha > 0$ (cf. \cite{Hermite2,Hermite1}). Expanding the rightmost side of \cref{generalizedhermite}:
\begin{align*}
\text{He}_{n}^{\alpha}(x) = \sum_{k \geq 0} \frac{(-\alpha)^{k}}{2^{k}(k)!} \partial_{x}^{2k}x^{n} = \sum_{k \geq 0} \frac{(n)!(-\alpha)^{k}}{2^{k}(k)!(n-2k)!}x^{n-2k} = \sum_{k \geq 0} \binom{n}{2k} \frac{(-\alpha)^{k}(2k)!}{2^{k}(k)!} x^{n-2k}.
\end{align*}
By comparing \cref{generalizedhermite} to repeated application of $h[-1]$ on $\textbf{1}$, setting $x=h(-1)$, $\alpha = 1/12$ and $1 = \textbf{1}$ above gives the expression of \Cref{Lemma: correctionlemma1}. Expressions of the form $h[-r]^{t} \textbf{1}$ for $r \geq 2$ in the round-bracket formalism are much more difficult to derive, however would prove extremely useful in computing further $p$-adic characters.
\end{rem}
\begin{lem}\label{Lemma: correctionlemma2}
For an odd integer $t \geq 1$, we have
\begin{align*}
h(1)h[-1]^{t}\textbf{1} = \sum_{k \geq 0} \binom{t}{2k+1} \frac{(2k+1)!}{k!(-24)^{k}} h(-1)^{t-2k-1}\textbf{1}.
\end{align*}
\end{lem}
\begin{proof}
The proof is similar to the computation of the second sum in \Cref{Lemma: correctionlemma1}.
\end{proof}
The following proposition is a slight generalization of Lemma 10.2 of \cite{FM} which outlines the family of square-bracket states which we will be considering throughout. \par 
\begin{prop}\label{Prop: 10.2generalization}
For odd integers $r,t \geq 1$, we have
\begin{align*}
(r-1)!h[-r]h[-1]^{t}\textbf{1} = \sum_{n \geq 0}\sum_{k \geq 0} &\binom{t}{2k} \frac{n!S_{r}^{(n+1)}(2k)!}{k!(-24)^{k}} h(-n-1)h(-1)^{t-2k}\textbf{1} \\
&- \sum_{k \geq 0} \binom{t}{2k+1} \frac{B_{r+1}(2k+1)!}{k!(r+1)(-24)^{k}} h(-1)^{t-2k-1}\textbf{1}
\end{align*}
where
\begin{align*}
S_{r}^{(n+1)} = \frac{1}{n!} \sum_{j \geq 0} \binom{n}{j} (-1)^{n+j} (j+1)^{r-1}
\end{align*}
are the Stirling numbers of the second kind.
\end{prop}
\begin{proof}
First, using \Cref{Lemma: correctionlemma1} we compute
\begin{multline*}
(r-1)!h[-r]h[-1]^{t}\textbf{1} \\ = (r-1)!\res{z}z^{-r} e^{z} \sum_{n \in \mathbb{Z}} h(n) \left( \sum_{k \geq 0} \binom{t}{2k} \frac{(2k)!}{k!(-24)^{k}}h(-1)^{t-2k}\textbf{1} \right) \left( e^{z} - 1 \right)^{-n-1}. 
\end{multline*}
Applying \Cref{Lemma: correctionlemma2} gives
\begin{multline*}
= (r-1)!\res{z}z^{-r} e^{z} \left( \sum_{n \geq 0} \sum_{k \geq 0} \binom{t}{2k} \frac{(2k)!}{k!(-24)^{k}} h(-n-1)h(-1)^{t-2k}\textbf{1} \left(e^{z}-1 \right)^{n} \right)  \\
+ (r-1)!\res{z}z^{-r} e^{z} \sum_{k \geq 0} \binom{t}{2k+1} \frac{(2k+1)!}{k!(-24)^{k}} h(-1)^{t-2k-1}\textbf{1} \left( e^{z} - 1 \right)^{-2}.
\end{multline*}
It was shown in \cite{FM} that
\begin{align*}
&(r-1)!\res{z}z^{-r}e^{z}(e^{z}-1)^{n} = n!S_{r}^{(n+1)} \\
&(r-1)!\res{z}z^{-r}e^{z}(e^{z}-1)^{-2} = - \frac{B_{r+1}}{r+1}
\end{align*}
and by comparing with the claim, this concludes the proposition.
\end{proof}

\subsection{Proof of Theorem 1.1} \label{Subsection: ProofofTheorem1.1}

We may now take the character of the family of states considered in \Cref{Prop: 10.2generalization} by using the following reformulation of equation (44) of \cite{MT} which gives an explicit description for the character of the state $a = h[-k_{1}]h[-k_{2}] \cdots h[-k_{r}]\textbf{1}$ with $k_{i} \geq 1$ under the map $\textbf{\textit{f}}$ encountered in \cref{padiccharactermap}:
\begin{align}
\textbf{\textit{f}}(a) = \sum_{(\ldots \lbrace s,t \rbrace \ldots) \in \mathcal{P}(\Phi,2)} \prod_{\lbrace s,t \rbrace}  \frac{2(-1)^{s+1}}{(s-1)!(t-1)!}G_{s+t}(q). \label{MTeq44}
\end{align}
Here, $\mathcal{P}(\Phi,2)$ denotes all possible partitions of the set $\Phi$ into parts $\lbrace s,t \rbrace$ of size 2. Then for each such partition, the product is taken over each part. \par
\vspace{10pt}
\begin{rem}
Here and throughout, we employ the normalization
\begin{align}
G_{k}(q) = -\frac{B_{k}}{2k} + \sum_{n \geq 1} \sigma_{k-1}(n)q^{n} \label{normalization}
\end{align}
seen in \cite{pSerre} and so the results of the papers cited here have been adjusted to this change. \par 
\end{rem}

\begin{prop}\label{Prop: specificcharactermap}
For odd integers $r,t \geq 1$, we have
\begin{align*}
\textbf{\textit{f}}\hspace{2pt} (h[-r]h[-1]^{t}\textbf{1}) = \frac{2^{(t+1)/2}t(t-2)!!}{(r-1)!}G_{2}^{(t-1)/2}(q)G_{r+1}(q).
\end{align*}
Here, $n!!$ denotes the product of all odd integers no greater than $n$.
\end{prop}
\begin{proof}
We make use of \cref{MTeq44}. In this case, $\Phi = \lbrace r, 1_{1},1_{2}, \ldots , 1_{t} \rbrace$ where we have labelled the $1$s for clarity. Since $t$ is odd, $\vert \Phi \vert$ is even and so $\mathcal{P}(\Phi, 2) \neq 0$. \par 
\vspace{10pt}
Suppose first that $r$ is paired with $1_{1}$, and the remaining $1_{2}, \ldots , 1_{t}$ are paired amongst themselves. The $t-1$ remaining $1$s then get put in $(t-1)/2$ parts of size $2$, hence
\begin{align*}
\left( \frac{2}{(r-1)!}G_{r+1}(q) \right) \left( \frac{2}{(1-1)!}G_{2}(q) \right)^{(t-1)/2} = \frac{2^{(t+1)/2}}{(r-1)!}G_{2}^{(t-1)/2}(q)G_{r+1}(q)
\end{align*}
where we recall that $r$ is odd here. There are then $(t-2)!!$ distinct ways of partitioning the remaining $1_{2}, \ldots , 1_{t}$ into $(t-1)/2$ parts of size $2$. Each way yields the same expression as above and so we multiply by $(t-2)!!$. Finally, we may repeat this process $t$ times, since $r$ can be paired with $1_{2}$ then with $1_{3}$ and so on, until $r$ is paired with $1_{t}$. This yields the expression
\begin{align*}
\frac{2^{(t+1)/2}t(t-2)!!}{(r-1)!}G_{2}^{(t-1)/2}(q)G_{r+1}(q)
\end{align*}
which establishes the result.
\end{proof}
\vspace{10pt}
With this, we denote (in the notation of \cite{FM}) the rescaled state
\begin{equation*}
u_{r,t} = (1-p^{r})(r-1)!h[-r]h[-1]^{t}\textbf{1}
\end{equation*}
for some prime $p > 3$. By \Cref{Prop: specificcharactermap}, 
\begin{equation}
\textbf{\textit{f}}(u_{r,t}) = (1-p^{r})(2^{(t+1)/2}t(t-2)!!)G_{2}^{(t-1)/2}(q)G_{r+1}(q). \label{characterofurt}
\end{equation}
For each fixed odd $t,l \geq 1$ where $l \not\equiv -1 \mod p-1$, we will assess the convergence of the sequence
\begin{equation}
(u_{p^{a}(p-1)+l,t})_{a \geq 0} = \left( (1-p^{p^{a}(p-1)+l})(p^{a}(p-1))!h[-p^{a}(p-1)-l]h[-1]^{t}\textbf{1} \right)_{a \geq 0} \label{sequenceofstates}
\end{equation}
in the $p$-adic topology. In order to do this, we require that the square-bracket states above be written in the round-bracket formalism. This was done in \Cref{Prop: 10.2generalization}. The following is then an extension of Lemma 10.4 of \cite{FM} and establishes the $p$-adic convergence of the terms involving Stirling numbers. \par 

\begin{lem}\label{Lemma: stirlingconvergence}
Fix $t,l \geq 1$ odd, and let $p$ be an odd prime with $r = p^{a}(p-1)+l$, $s=p^{b}(p-1)+l$ and $a \leq b$. Then for any fixed $k$ within the range $0 \leq k \leq \lfloor t/2 \rfloor$ and any $n \geq 0$ we have
\begin{align*}
(1-p^{r})\binom{t}{2k} \frac{n!S_{r}^{(n+1)}(2k)!}{k!(-24)^{k}} \equiv (1-p^{s})\binom{t}{2k} \frac{n!S_{s}^{(n+1)}(2k)!}{k!(-24)^{k}} \mod p^{a+x+1}.
\end{align*}
for some fixed integer $x$.
\end{lem}
\begin{proof}
Let $\vert \cdot \vert_{p}$ denote the $p$-adic absolute value. We have
\begin{align*}
\left\lvert \binom{t}{2k} \frac{n!S_{r}^{(n+1)}(2k)!}{k!(-24)^{k}} - \binom{t}{2k} \frac{n!S_{s}^{(n+1)}(2k)!}{k!(-24)^{k}} \right\rvert_{p} = \left\lvert n!S_{r}^{(n+1)} - n!S_{s}^{(n+1)} \right\rvert_{p} \left\lvert \binom{t}{2k} \frac{(2k)!}{k!(-24)^{k}} \right\rvert_{p}.
\end{align*}
The rightmost term is dependent only on $t$ and $k$ which are fixed, and so denote by $x$ the $p$-adic valuation of this term, which is also fixed. We show now that $n!S_{r}^{(n+1)} \equiv n!S_{s}^{(n+1)} \mod p^{a+1}$. From the statement of \Cref{Prop: 10.2generalization}, we have the formula
\begin{align*}
n!S_{r}^{(n+1)} = \sum_{j = 0}^{n} \binom{n}{j} (-1)^{n+j} (j+1)^{r-1}.
\end{align*}
If $p \mid (j+1)$ then of course $(j+1)^{r-1} \equiv 0 \mod p^{a+1}$. Suppose then, that $p \nmid (j+1)$. This means that $p$ and $(j+1)$ are coprime, and subsequently that $p^{a+1}$ and $(j+1)$ are coprime. Recall that $\varphi (p^{a+1}) = p^{a}(p-1)$ where $\varphi$ denotes the Euler totient function. So by Euler's theorem,
\begin{align*}
(j+1)^{p^{a}(p-1)} \equiv 1 \mod p^{a+1}.
\end{align*}
Note also that $(j+1)^{l-1}$ is an integer independent of $r$ and so by multiplying both sides of the above congruence by this factor we get $(j+1)^{r-1} \equiv (j+1)^{l-1} \mod p^{a+1}$. Thus
\begin{align*}
n!S_{r}^{(n+1)} \equiv \sum_{\substack{ j=0 \\ p \nmid (j+1) }}^{n} \binom{n}{j} (-1)^{n+j} (j+1)^{l-1} \mod p^{a+1}
\end{align*}
and since the right hand side above does not depend on $r$, this establishes that $n!S_{r}^{(n+1)} \equiv n!S_{s}^{(n+1)} \mod p^{a+1}$. Putting everything together, we get
\begin{align*}
\binom{t}{2k} \frac{n!S_{r}^{(n+1)}(2k)!}{k!(-24)^{k}} \equiv \binom{t}{2k} \frac{n!S_{s}^{(n+1)}(2k)!}{k!(-24)^{k}} \mod p^{a+x+1}.
\end{align*}
Finally since $x$ is fixed, for sufficently large $a$ it is clear that
\begin{align*}
1-p^{p^{a}(p-1)+l} \equiv 1 - p^{p^{b}(p-1)+l} \mod p^{a+x+1},
\end{align*}
and so by combining the above two congruences together, we obtain the desired result.
\end{proof}
Next we establish the $p$-adic convergence of the terms involving the Bernoulli numbers. This is done in the following lemma by use of the classical Kummer congruences (cf. \cite{Kummer} for the statement and proof).
\begin{lem}\label{Lemma: bernoulliconvergence}
Fix $t,l \geq 1$ odd where $l \not\equiv -1 \mod p-1$, and let $p \geq 5$ be a prime with $r = p^{a}(p-1)+l$, $s=p^{b}(p-1)+l$ and $a \leq b$. Then for $k$ in the range $0 \leq k \leq \lfloor (t-1)/2 \rfloor$, we have
\begin{align*}
(1-p^{r}) \binom{t}{2k+1} \frac{B_{r+1}(2k+1)!}{k!(r+1)(-24)^{k}} \equiv (1-p^{s})\binom{t}{2k+1} \frac{B_{s+1}(2k+1)!}{k!(s+1)(-24)^{k}} \mod p^{a+y+1}
\end{align*}
for some fixed integer $y$.
\end{lem}
\begin{proof}
First write
\begin{align*}
\left\lvert (1-p^{r})\frac{B_{r+1}}{r+1} - (1-p^{s})\frac{B_{s+1}}{s+1} \right\rvert_{p} \left\lvert \binom{t}{2k+1} \frac{(2k+1)!}{k(-24)^{k}} \right\rvert_{p}.
\end{align*}
Once again the rightmost term is dependent only on $t$ and $k$ which are fixed, and so denote by $y$ the $p$-adic valuation of this term, which is also fixed. Notice that
\begin{align*}
p^{a}(p-1) + l+1 \equiv p^{b}(p-1) + l+1 \mod p^{a}(p-1)
\end{align*}
and $p^{a}(p-1) = \varphi (p^{a+1})$. Since $l$ is odd, both $r+1$ and $s+1$ are even. Since $l \not\equiv -1 \mod p-1$, $p-1$ does not divide $l+1$ and so both $r+1$ and $s+1$ are not divisible by $p-1$ (since $p > 3$). Thus by Kummer's congruence we obtain
\begin{align*}
\left\lvert (1-p^{r})\frac{B_{r+1}}{r+1} - (1-p^{s})\frac{B_{s+1}}{s+1} \right\rvert_{p}\left\lvert \binom{t}{2k+1} \frac{(2k+1)!}{k(-24)^{k}} \right\rvert_{p} = \frac{1}{p^{a+y+1}}
\end{align*}
which gives us the right congruence.
\end{proof}
We are now in a position to prove \cref{Th: theorem1}. The congruences established in both \Cref{Lemma: stirlingconvergence,Lemma: bernoulliconvergence} imply that
\begin{align*}
\lim_{a \to \infty} u_{p^{a}(p-1)+l,t} = u_{l,t}
\end{align*}
for some state $u_{l,t}$ in the $p$-adic Heisenberg VOA $S_{1}$. Then in the $p$-adic topology and from \cref{characterofurt}, we have
\begin{align*}
\textbf{\textit{f}}(u_{l,t}) &= \lim_{a \to \infty} (1-p^{p^{a}(p-1)+l})2^{(t+1)/2}t(t-2)!!G_{2}^{(t-1)/2}(q)G_{p^{a}(p-1)+l+1}(q) \\
&= 2^{(t+1)/2}t(t-2)!! G_{2}^{(t-1)/2}(q)\lim_{a \to \infty} (1-p^{p^{a}(p-1)+l+1})G_{p^{a}(p-1)+l+1}(q).
\end{align*}
Since $\lim_{a \to \infty} 1-p^{p^{a}(p-1)+l+1} = 1$ and $\lim_{a \to \infty} p^{a}(p-1)+l+1 = l+1$,
\begin{align*}
\textbf{\textit{f}}(u_{l,t}) = 2^{(t+1)/2}t(t-2)!!G_{2}^{(t-1)/2}(q)G_{l+1}^{\star}(q)
\end{align*}
where $G_{l+1}^{\star}(q)$ is a $p$-adic Eisenstein modular form of weight $l+1$. This proves \cref{Th: theorem1}. \par

\section{P-adic characters in lattice VOAs} \label{Section: padiccharactersinlatticevoas}

\subsection{Lattice VOAs}

We recount relevant facts about lattice VOAs $V_{\Lambda}$ which are necessary in the computation of characters in the subsection below. In the context of \cite{DMN} or \cite{LL}, let $\Lambda$ be an even (integral) unimodular lattice of rank $d$ with associated symmetric $\mathbb{Z}$-bilinear form $\langle \cdot , \cdot \rangle$, and let
\begin{align*}
\mathfrak{h} = \Lambda \otimes_{\mathbb{Z}} \mathbb{C}.
\end{align*}
Following the process outlined in \Cref{Subsection: HeisenbergandSquareBracket}, construct $S_{\alg}$ where we denote by $a(n)$ the action of $a \otimes t^{n}$ on $S_{\alg}$ for $a \in \mathfrak{h}$, and where
\begin{align*}
[a(m),b(n)] =  \delta_{m+n,0} \langle a,b \rangle m.
\end{align*}
for integers $m,n \in \mathbb{Z}$. Denote by $\mathbb{C}\lbrace \Lambda \rbrace$ the group algebra with basis $\lbrace e^{\alpha} \mid \alpha \in \Lambda \rbrace$ and form the space
\begin{align*}
V_{\Lambda} = S_{\alg} \otimes \mathbb{C}\lbrace \Lambda \rbrace.
\end{align*}
Vertex operators of states in $V_{\Lambda}$ with trivial $\mathbb{C}\lbrace \Lambda \rbrace$ tensor factor are defined in the same way as in $S_{\alg}$. We identify $S_{\alg}$ with $S_{\alg} \otimes e^{0}$ hence the vacuum of $V_{\Lambda}$ is the same as that of $S_{\alg}$. On $V_{\Lambda}$, states $a \in \mathfrak{h}$ with $n \neq 0$ act on $S_{\alg}$ as before:
\begin{align*}
a(n)(b \otimes e^{\alpha}) = a(n)b \otimes e^{\alpha},
\end{align*}
whereas the operator $a(0)$ acts on $\mathbb{C}\lbrace \Lambda \rbrace$ as
\begin{align}
a(0)(b \otimes e^{\alpha}) = \langle a, \alpha \rangle (b \otimes e^{\alpha}). \label{zeroproductonlattice}
\end{align}
Define the dual lattice $\Lambda^{\circ} = \Set{ \alpha \in \mathfrak{h} | \langle \alpha, \Lambda \rangle \subset \mathbb{Z} }$. Since $\Lambda$ is integral, $\Lambda \subset \Lambda^{\circ}$ and so we may consider the coset decomposition
\begin{align*}
\Lambda^{\circ} \bigcup_{i \in \Lambda^{\circ} / \Lambda} (\Lambda + \lambda_{i})
\end{align*}
From \cite{Dong}, each space
\begin{align*}
V_{\Lambda + \lambda_{i}} = S_{\alg} \otimes \mathbb{C}\lbrace \Lambda + \lambda_{i} \rbrace
\end{align*}
constitute the complete set of non-isomorphic irreducible $V_{\Lambda}$-modules, and furthermore, a theorem of Zhu shows that each $Z_{\Lambda + \lambda_{i}}(a,q)$ is modular. \par 
\vspace{10pt}
It is easy to show (\cite{DMN}, Lemma 4.2.1) that $Z_{\Lambda + \lambda_{i}}(a,q) = 0$ for any $a = S_{\alg} \otimes e^{\alpha} \in V_{\Lambda}$ for some non-zero $\alpha \in \Lambda$. The proof follows from the fact that any state in $V_{\Lambda}$ with non-zero $\mathbb{C}\lbrace \Lambda \rbrace$ tensor factor has no mode which preserves grading. As such we will consider only characters of states of the form $S_{\alg} \otimes e^{0} \in V_{\Lambda}$, and we henceforth omit writing the trivial $\mathbb{C}\lbrace \Lambda \rbrace$ tensor factor for such states. For each $\alpha \in \Lambda + \lambda_{i}$, we have
\begin{align}
V_{\Lambda + \lambda_{i}} = \bigoplus_{\alpha \in \Lambda + \lambda_{i}} S_{\alg} \otimes e^{\alpha} \label{voacosetdecomposition}
\end{align}
and so following \cite{DMN}, denote for $a \in V_{\Lambda}$ the functions
\begin{align*}
&Z_{\alpha}(a,q) = Z_{S_{\alg} \otimes \iota (e_{\alpha})}(a,q) \\
&Z_{i}(a,q) = Z_{\Lambda + \lambda_{i}}(a,q).
\end{align*}
It is then clear from \cref{voacosetdecomposition} that
\begin{align}
Z_{i}(a,q) = \sum_{\alpha \in \Lambda + \lambda_{i}} Z_{\alpha}(a,q). \label{cosetcharactersum}
\end{align}
Computing characters in $V_{\Lambda}$ is similar to that in $S_{\alg}$, however one must be cautious due to discrepancies caused by the behaviour of $a(0)$ exposed in \cref{zeroproductonlattice} above. In particular, equation \cref{MTeq44} of \cite{MT} does not apply for most states in $V_{\Lambda}$. We resort then to using the following general recursive expression given first in \cite{Zhu} from which \cref{MTeq44} is based: Given any $V$-module $M$ and for $a,b \in V$, we have
\begin{align}
Z_{M}(a[-n]b,q) = \delta_{n,1} \text{Tr}_{M} \left( o(a)o(b)q^{L_{0}- c_{V}/24} \right) + \sum_{m \geq 1} \frac{2(-1)^{m+1}}{m!(n-1)!}G_{n+m}(q)Z_{M}(a[m]b,q) \label{zhurecursion}
\end{align}
where we have again applied the normalization seen in \cref{normalization}.

\subsection{Proof of Theorem 1.2}  \label{Subsection: ProofofThm1.2}

We take $a \in V_{\Lambda}$ such that $\langle a, a \rangle = 1$ and $\alpha \in \Lambda^{\circ}$ throughout. We will consider as in \Cref{Subsection: somesquarebracketstates} the states 
\begin{align*}
v_{r,t} = (r-1)!a[-r]a[-1]^{t}\textbf{1} \in S_{\alg} \otimes e^{0}
\end{align*}
where $r,t$ are odd integers, and where $t \geq 1$ and $r \geq 3$. \par 
\vspace{10pt}
In order to assess convergence later, we must re-write $v_{r,t}$ in the round-bracket formalism. This will be similar to what was done in \Cref{Prop: 10.2generalization}, however from \cref{zeroproductonlattice}, we must be careful with the product $a(0)$. Recall that from \cref{squarebracketh-1} we have
\begin{align*}
a[-1] = a(-1) + \frac{1}{2}a(0) - \frac{1}{12}a(-1) + \cdots.
\end{align*}
We wish to find expressions for $a[-1]^{t}\textbf{1}$ and $a(1)a[-1]^{t}\textbf{1}$. Since both are elements of $S_{\alg} \otimes e^{0}$ however, the product $a(0)$ in the expansion above acts as zero since $\langle a, 0 \rangle = 0$. We will thus obtain the same expressions for $a[-1]^{t}\textbf{1}$ and $a(1)a[-1]^{t}\textbf{1}$ as the ones in the Heisenberg case given in \Cref{Lemma: correctionlemma1,Lemma: correctionlemma2} with $h=a$. For the same reason, the statement of \Cref{Prop: 10.2generalization} holds in in this case as well:
 
\begin{prop} \label{Prop: roundbracketlatticestates}
Let $a \in V_{\Lambda}$ where $\langle a, a \rangle = 1$ and $\alpha \in \Lambda^{\circ}$. For $r,t$ odd integers and $t \geq 1$ and $r \geq 3$, we have
\begin{align*}
v_{r,t} =  \sum_{n \geq 0}\sum_{k \geq 0} &\binom{t}{2k} \frac{n!S_{r}^{(n+1)}(2k)!}{k!(-24)^{k}} a(-n-1)a(-1)^{t-2k}\textbf{1} \\
&- \sum_{k \geq 0} \binom{t}{2k+1} \frac{B_{r+1}(2k+1)!}{k!(r+1)(-24)^{k}} a(-1)^{t-2k-1}\textbf{1}
\end{align*}
\end{prop}

We now take the character of $v_{r,t}$ over $S_{\alg} \otimes e^{\alpha}$. The following is an extension of Lemma 4.2.2 of \cite{DMN} and is proven similarly.

\begin{lem} \label{Lemma: DMN4.2.2}
Let $a \in V_{\Lambda}$ where $\langle a, a \rangle = 1$ and $\alpha \in \Lambda^{\circ}$. For $t \geq 0$ we have
\begin{align*}
Z_{\alpha}(a[-1]^{t}\textbf{1},q) = \left( \sum_{k \geq 0} \binom{t}{2k} \frac{(2k)!}{k!} \langle a, \alpha \rangle^{t-2k} G_{2}(q)^{k} \right) q^{\frac{1}{2} \langle \alpha, \alpha \rangle} / \eta(q)^{d}.
\end{align*}
\end{lem}
\begin{proof}
We proceed by induction on $t$. The case $t=0$ is simply the graded dimension over $S_{\alg} \otimes e^{0}$, and so let $t=1$. Then using \cref{zhurecursion},
\begin{align*}
Z_{\alpha}(a[-1]\textbf{1},q) &= \langle a, \alpha \rangle Z_{\alpha}(\textbf{1},q) = \langle a, \alpha \rangle q^{\frac{1}{2}} / \eta (q)^{d}
\end{align*}
where we used the fact that $o(a) = \langle a, \alpha \rangle$ since we work over $S_{\alg} \otimes e^{\alpha}$. Suppose the claim holds up to $t-1$. Then using \cref{zhurecursion} again, we obtain
\begin{align}
Z_{\alpha}(a[-1]^{t}\textbf{1}, q) &= Z_{\alpha}(a[-1]a[-1]^{t-1}\textbf{1},q) \nonumber \\
&= \langle a, \alpha \rangle Z_{\alpha}(a[-1]^{t-1}\textbf{1},q) + 2(t-1)G_{2}(q)Z_{\alpha}(a[-1]^{t-2}\textbf{1},q). \label{DMN4.2.2eq1}
\end{align}
By induction, we write the first term of \cref{DMN4.2.2eq1} as
\begin{align}
\langle a, \alpha \rangle^{t} + \left( \sum_{k \geq 1} \binom{t-1}{2k} \frac{(2k)!}{k!} \langle a, \alpha \rangle^{t-2k}G_{2}^{k}(q) \right)q^{\frac{1}{2}\langle \alpha, \alpha \rangle} / \eta (q)^{d}, \label{DMN4.2.2eq2}
\end{align}
and similarly the second term of \cref{DMN4.2.2eq1} equates to
\begin{align}
&\left( \sum_{k \geq 0} \binom{t-2}{2k} \frac{2(t-1)(2k)!}{k!} \langle a, \alpha \rangle^{t-2k-2} G_{2}(q)^{k+1} \right) q^{\frac{1}{2}\langle \alpha, \alpha \rangle} / \eta (q)^{d} \nonumber \\
&= \left( \sum_{k \geq 0} \frac{(k+1)(2k+1)!(2)(t-1)!}{(k+1)(2k+1)!(t-2k-2)!(k)!} \langle a, \alpha \rangle^{t-2k-2} G_{2}(q)^{k+1} \right) q^{\frac{1}{2}\langle \alpha, \alpha \rangle} / \eta (q)^{d} \nonumber \\
&= \left( \sum_{k \geq 0} \binom{t-1}{2k+1} \frac{(2k+2)!}{(k+1)!} \langle a, \alpha \rangle^{t-2k-2} G_{2}(q)^{k+1} \right) q^{\frac{1}{2}\langle \alpha, \alpha \rangle} / \eta (q)^{d} \nonumber \\
&= \left( \sum_{k \geq 1} \binom{t-1}{2k-1} \frac{(2k)!}{k!} \langle a, \alpha \rangle^{t-2k} G_{2}(q)^{k} \right) q^{\frac{1}{2}\langle \alpha, \alpha \rangle} / \eta (q)^{d} \label{DMN4.2.2eq3}.
\end{align}
Summing \cref{DMN4.2.2eq2,DMN4.2.2eq3} and using Pascal's rule as in \Cref{Lemma: correctionlemma1}, we obtain
\begin{align*}
Z_{\alpha}(a[-1]^{t}\textbf{1},q) = \left( \sum_{k \geq 0} \binom{t}{2k} \frac{(2k)!}{k!} \langle a, \alpha \rangle^{t-2k} G_{2}(q)^{k} \right) q^{\frac{1}{2} \langle \alpha, \alpha \rangle} / \eta(q)^{d}
\end{align*}
which proves the result.
\end{proof}

Following \cite{DMN}, we will be denoting
\begin{align*}
f_{\alpha,t}(q) = \sum_{k \geq 0} \binom{t}{2k} \frac{(2k)!}{k!} G_{2}(q)^{k} \langle a, \alpha \rangle^{t-2k} 
\end{align*}
in the following analogue of \Cref{Prop: specificcharactermap}:

\begin{prop} \label{Prop: specificcharactermap2}
Let $a \in V_{\Lambda}$ where $\langle a, a \rangle = 1$. For $r,t$ odd integers and $t \geq 1$ and $r \geq 3$ we have
\begin{align*}
Z(v_{r,t},q) = 2tG_{r+1}(q) \left( \sum_{k \geq 0} \binom{t-1}{2k} \frac{(2k)!}{k!} G_{2}(q)^{k} \Theta_{\Lambda, t-2k-1}(q) \right) / \eta (q)^{d}
\end{align*}
where
\begin{align*}
\Theta_{\Lambda, t-2k-1}(q) = \sum_{\alpha \in \Lambda} \langle a, \alpha \rangle^{t-2k-1} q^{\frac{1}{2} \langle \alpha, \alpha \rangle}.
\end{align*}
\end{prop}
\begin{proof}
We use \cref{zhurecursion}. This gives
\begin{align*}
Z_{\alpha}(v_{r,t},q) &= (r-1)!\sum_{m \geq 1} \frac{2(-1)^{m+1}}{m!(r-1)!}G_{r+m}(q)Z_{\alpha}(a[m]a[-1]^{t}\textbf{1},q) \\
&= 2G_{r+1}(q)Z_{\alpha}(a[1]a[-1]^{t}\textbf{1},q) \\
&= 2tG_{r+1}(q)f_{\alpha, t-1}(q)q^{\frac{1}{2} \langle \alpha, \alpha \rangle} / \eta(q)^{d}.
\end{align*}
where we have made use of \Cref{Lemma: DMN4.2.2} in the last equality. Using \cref{cosetcharactersum},
\begin{align*}
Z_{i}(v_{r,t},q) &= 2tG_{r+1}(q) \left( \sum_{\alpha \in \Lambda + \lambda_{i}} f_{\alpha, t-1}(q)q^{\frac{1}{2} \langle \alpha, \alpha \rangle} \right) / \eta(q)^{d} \\
&= 2tG_{r+1}(q) \left( \sum_{k \geq 0} \binom{t-1}{2k} \frac{(2k)!}{k!} G_{2}(q)^{k} \sum_{\alpha \in \Lambda + \lambda_{i}} \langle a, \alpha \rangle^{t-2k-1} q^{\frac{1}{2}\langle \alpha, \alpha \rangle} \right) / \eta(q)^{d} \\
&= 2tG_{r+1}(q) \left( \sum_{k \geq 0} \binom{t-1}{2k} \frac{(2k)!}{k!} G_{2}(q)^{k} \Theta_{\Lambda + \lambda_{i}, t-2k-1}(q) \right) / \eta (q)^{d}.
\end{align*}
The proof follows from the fact that specializing $i=0$ gives $Z_{0}(v_{r,t},q) = Z(v_{r,t},q)$.
\end{proof}

The theta series $\Theta_{\Lambda, t,k}(q)$ in the statement of \Cref{Prop: specificcharactermap2} appears troublesome since it is not quite the regular theta series $\Theta_{\Lambda}(q)$ of $\Lambda$. It was shown in \cite{DM} and exposed in Theorem 4.2.6 of \cite{DMN} that $\Theta_{\Lambda, t,k}(q)$ in fact lies in the space of quasi-modular forms on the congruence subgroup $\Gamma_{0}(N)$ for some $N \geq 1$. Hence the image of the character map for $V_{\Lambda}$ is indeed quasi-modular, and we may consider $p$-adic limits. \par 
\vspace{10pt}
With the assumptions of \Cref{Prop: specificcharactermap2} in place, we proceed as in \Cref{Subsection: ProofofTheorem1.1} by denoting 
\begin{align*}
u_{r,t} = (1-p^{r})v_{r,t} = (1-p^{r})(r-1)!a[-r]a[-1]^{t}\textbf{1}
\end{align*}
and considering the sequence of states $\left( u_{p^{b}(p-1)+l,t} \right)_{b \geq 0}$ for fixed odd $t \geq 1$ where $p \geq 3$ is a prime and $l \not\equiv -1 \mod p-1$. Since the round-bracket states obtained in \Cref{Prop: roundbracketlatticestates} are identical to those in the Heisenberg case, both \Cref{Lemma: stirlingconvergence,Lemma: bernoulliconvergence} establish the $p$-adic convergence of this sequence to a state $u_{t,l}$ in the $p$-adic lattice VOA $\widehat{V}_{\Lambda}$. Then by taking the limit in the $p$-adic topology as in \Cref{Subsection: ProofofTheorem1.1} we obtain
\begin{align*}
\textbf{\textit{f}}(u_{t,l}) = 2tG_{l+1}^{\star}(q) \left( \sum_{k \geq 0} \binom{t-1}{2k} \frac{(2k)!}{k!} G_{2}(q)^{k} \Theta_{\Lambda, t-2k-1}(q) \right).
\end{align*}
where $\textbf{\textit{f}}$ is the character map in \cref{padiccharactermaplattice}. This proves \Cref{Th: theorem2}. \par

\bibliographystyle{amsalpha}
\bibliography{references}

\end{document}